\documentclass[11pt]{amsart}

\usepackage[text={6.5in,9in},centering]{geometry}
\usepackage{amsmath}
\usepackage{amssymb}
\usepackage{amsfonts}
\usepackage{amscd}
\usepackage[pdftex]{graphicx}
\usepackage{verbatim}
\usepackage{color}
\usepackage{subfig}
\usepackage{bm} 

\newtheorem{theorem}{Theorem}[section]
\newtheorem{lemma}[theorem]{Lemma}
\newtheorem{proposition}[theorem]{Proposition}

\newtheorem{remark}[theorem]{Remark}
\numberwithin{equation}{section} \numberwithin{table}{section}

\newcommand{\bbf}{\bm{f}}
\newcommand{\bx}{\bm{x}}
\newcommand{\by}{\bm{y}}

\newcommand{\bu}{\bm{u}}
\newcommand{\bv}{\bm{v}}
\newcommand{\bw}{\bm{w}}

\newcommand{\be}{\bm{e}}

\newcommand{\one}{\bm{1}}

\newcommand{\R}{\mathbb{R}}

\def\Re{I\!\!R}

\newcommand{\VV}{\mathcal{V}}
\newcommand{\EE}{\mathcal{E}}
\newcommand{\MM}{\mathcal{M}}

\newcommand{\GG}{\mathcal{G}}
\newcommand{\abs}[1]{\lvert#1\rvert}

\begin{document}

\title[Multilevel preconditioning based on matching of
graphs]{Algebraic multilevel preconditioners for the graph {L}aplacian
  based on matching in graphs}
\thanks{This work supported in part by the National Science
Foundation, DMS-0810982, OCI-0749202 and by the Austrian Science Fund, Grants P19170-N18 and P22989-N18.}

\author{J.~Brannick}
\address{Department of Mathematics, The Pennsylvania State University,
University Park, PA 16802, USA}
\email{brannick@psu.edu}

\author{Y.~Chen}
\address{Department of Mathematics, The Pennsylvania State University,
University Park, PA 16802, USA}
\email{chen\_y@math.psu.edu}
 
\author{J.~Kraus}
\address{Johann Radon Institute, Austrian Academy of Sciences,
Altenberger Str.~69, 4040 Linz, Austria}
\email{johannes.kraus@oeaw.ac.at}
 
\author{L.~Zikatanov}
\address{Department of Mathematics, The Pennsylvania State University,
University Park, PA 16802, USA}
\email{ludmil@psu.edu}
 
\date{Today is \today}
 

\begin{abstract}
This paper presents estimates of the convergence rate and complexity of an 
algebraic multilevel preconditioner based on piecewise constant coarse vector spaces 
applied to the graph Laplacian. 
A bound is derived on the energy norm of the projection operator 
onto any piecewise constant vector space, 
which results in
an estimate of the two-level convergence rate 
where the coarse level graph is obtained by matching.  
The two-level convergence of the method is then used to establish 
the convergence of an Algebraic Multilevel Iteration that uses the two-level scheme recursively.
On structured grids, the method is proven to have 
convergence rate $\approx (1-1/\log n)$ 
and $O(n\log n)$ complexity for each cycle, where $n$ denotes the number of unknowns
in the given problem. 
Numerical results of the algorithm applied to various 
graph Laplacians are reported.
It is also shown that all the theoretical estimates derived for matching can be generalized to the case of
 aggregates containing more than two vertices.  
\end{abstract}

\maketitle

\section{Introduction}
Algebraic Multigrid (AMG) attempts to mimic the main components of Geometric Multigrid
in an algebraic fashion, that is, by using information from the coefficient
matrix only to construct the multilevel solver.  The basic algorithm uses a setup phase to construct a nested sequence of coarse spaces
that are then used in the solve phase to compute the solution.  
The two main approaches to the AMG setup algorithm are classical AMG~\cite{geo,BMR83} and (smoothed) aggregation AMG~\cite{SA_61,MV92,MBrezina_etal_1999a,PVanek_JMandel_MBrezina_1995a,2003KimH_XuJ_ZikatanovL-aa}, which are distinguished by the type of coarse variables used in the construction of AMG interpolation.  

In the classical AMG algorithm, the coarse variables are chosen using a coloring algorithm which is designed to find a suitable maximal independent subset 
of the fine variables.   Then, given the coarse degrees of freedom, a row of interpolation is constructed for each fine point from its neighboring coarse points.  In contrast, the aggregation-based AMG setup algorithm partitions the fine variables into disjoint subdomains, called aggregates.  Then, a column (or several columns as in~\cite{PVanek_JMandel_MBrezina_1995a}) of interpolation is associated to each aggregate, which has nonzero entries only for the unknowns belonging to this aggregate.  The focus of this paper is on the development and, in particular, the analysis of the latter aggregation-type methods.

The idea of aggregating unknowns to coarsen a system of discretized partial differential equations dates back to work by Leont'ev in 1959~\cite{Leo59}.
Simon and Ando developed a related technique for aggregating dynamic systems in 1961 \cite{SA_61} and 
a two-grid aggregation-based scheme was considered in the context of solving Markov chain systems by Takahashi in 1975~\cite{YT_1975}.
Aggregation-based methods have been studied extensively since and numerous algorithms and theoretical results have followed~\cite{MBrezina_etal_1999a,PVanek_JMandel_MBrezina_1995a,2003KimH_XuJ_ZikatanovL-aa}.   Van\'{e}k introduced an extension of these methods known as smoothed aggregation multigrid in which smoothing steps
are applied to the columns of the aggregation-based interpolation operator to accelerate two-level convergence and a modification of this two-level algorithm with overcorrection is presented in~\cite{MV92}.  
A multilevel smoothed aggregation algorithm and its convergence analysis are found in~\cite{PVanek_MBrezina_JMandel_2001a} and, in~\cite{panayot-book}, an improved convergence theory of the method is presented.  The latter theory is then extended to allow for aggressive coarsening, provided an appropriate polynomial smoother is used~\cite{BVV_2011}.  A further generalization known as adaptive smoothed  aggregation is developed in~\cite{MBrezina_etal_2004a}.
Variants of the above approaches continue to be developed for use in scientific computing and have been developed for higher order partial differential equations \cite{PVanek_JMandel_MBrezina_1995a},  convection diffusion problems~\cite{2003KimH_XuJ_ZikatanovL-aa}, Markov chains \cite{HDeSterck_08,brannick_markov_2011}, and the Dirac equation in quantum chromodynamics~\cite{brannick_clark_brower_osborne_rebbi_2008}.

In this paper, an aggregation based Algebraic Multigrid method for the graph Laplacian is presented. 
The approach constructs the sequence of coarse graphs recursively using a
pair-wise aggregation, or matching, form of interpolation.  
However, 
it is demonstrated here, that the convergence rate of a two-level method based on such a construction
is uniformly bounded for the graph Laplacian on general graphs and, thus, can be used within an Algebraic 
Multilevel Iteration (AMLI) \cite{amli-1, panayot-book}  
as a preconditioner to the Conjugate Gradient iteration to obtain a nearly optimal solver. 
A noteworthy feature of the approach is its simplicity, which makes it possible to analyze the
convergence and complexity of the method with few assumptions and without any geometric 
information. 

The remainder of the paper is organized as follows. 
In Section \ref{problem-description}, we introduce the graph Laplacian problem
and discuss some of its applications. 
In Section \ref{matching}, we introduce a graph matching algorithm and demonstrate 
that the energy norm of the $\ell_{2}$ projection onto the coarse space is a key quantity
in deriving convergence and complexity estimates of the method.  Additionally,  
we introduce an approach 
computing an approximation of the energy norm of this projection operator. 
In Section \ref{twolevel}, we present an analysis on the two-level method for
the graph Laplacian operator. 
In Section \ref{amli}, we consider the convergence and complexity of the resulting AMLI method, 
and in Section \ref{numerics} we provide numerical results and address some 
practical issues of the method.  

\section{Problem formulation and notation}\label{problem-description}

Consider an unweighted connected graph $\GG=(\VV,\EE)$, 
where $\VV$ denotes the set of vertices
and $\EE$ denotes the set of edges of $\GG$. 
The variational problem considered here is as follows: 
Given an $\bbf$ satisfying $(\bbf,\one)=0$, 
where $\one$ is a constant vector, 
find a $\bu\in \R^{n}$, 
where $n=\abs{\VV}$ denotes the cardinality of the set of vertices $\VV$, 
such that  

\begin{eqnarray}\label{varprob}
(A \bu, \bv) = (\bbf, \bv)
, \quad
\forall \bv \in \R^{n},
\end{eqnarray}
where
\begin{eqnarray}
(A \bu, \bv) 
=
\sum_{k=(i,j)\in\EE} (\bu_{i}-\bu_{j})(\bv_{i}-\bv_{j}) 
, \qquad
(\bbf, \bv)
=
\sum_{i\in \VV} \bbf_{i}\bv_{i}
\text{ , }
(\bbf, \one)
=
\sum_{i\in \VV} \bbf_{i}
.
\end{eqnarray}
Define the discrete gradient operator 
$B: \R^{|\VV|} \mapsto \R^{|\EE|}$ 
such that 
\begin{eqnarray*}
(B\bu)_{k} 
=\bu_{i}-\bu_{j}
,\quad 
k=(i,j)\in \EE
,\quad 
i<j ,
\end{eqnarray*}
where $e_{i}$ and $e_{j}$ are standard Euclidean bases. 
The operator $A$ is named graph {\em Laplacian} since 
\begin{eqnarray*}
A=B^{T}B. 
\end{eqnarray*}

The operator $A$ is symmetric and positive semi-definite
and its kernel is the space spanned by the constant vector. 
These properties can also be verified by the matrix form of $A$ 
defined in the following way
\begin{eqnarray*}
(A)_{ij}=
\left\{
\begin{array}{ll}
d_{i} \qquad & i=j; \\
-1  & i\neq j, (i,j)\in \EE; \\
0  & i\neq j, (i,j)\notin \EE ;
\end{array}
\right.
\end{eqnarray*}
where $d_{i}$ is the degree of the $i$-th vertex. 

Efficient multilevel graph Laplacian solvers 
are important in numerous application areas, including finite element and finite difference
discretizations of elliptic partial differential equations (PDE), data mining, clustering in images, and 
as a preconditioner for weighted graph Laplacians. 
Moreover, 
the theory developed here for multilevel aggregation solvers applied to graph Laplacians should provide 
insights on how to design a solver for more general weighted graph Laplacians, which cover also
anisotropic diffusion problems. 
Two generalizations of the graph Laplacian systems are as follows.
\begin{itemize}

\item
{\em Weighted graph Laplacians: } 
Assume that the graph is weighted and 
the $k$-th edge is assigned a weight $w_{k}$, 
then the corresponding bilinear form of $A$ is 
\begin{eqnarray*}
(A \bu, \bv) 
=
\sum_{k=(i,j)\in\EE} w_{k}(\bu_{i}-\bu_{j})(\bv_{i}-\bv_{j}) .
\end{eqnarray*}
Define $D: \R^{|\EE|} \mapsto \R^{|\EE|}$ as 
a diagonal matrix whose $k$-th diagonal entry is equal to $w_{k}$, 
then the matrix $A$ can be decomposed as 
\begin{eqnarray*}
A=B^{T}DB .
\end{eqnarray*}
Finite element and finite difference discretizations of 
elliptic PDEs with Neumann boundary conditions results in such weighted graph 
Laplacians. 

\item
{\em Positive definite matrices: }
Assume that $A$ is defined in terms of the bilinear form 
\begin{eqnarray*}
(A \bu, \bv) 
=
\sum_{k=(i,j)\in\EE} (\bu_{i}-\bu_{j})(\bv_{i}-\bv_{j}) + 
\sum_{i\in\VV}\bu_{i}\bv_{i}
=(A_{s}\bu, \bv)+(A_{t}\bu, \bv).
\end{eqnarray*}
By introducing a Lagrange multiplier
$y$, 
the system $A\bu=\bbf$ can be rewritten as an augmented linear system 
\begin{eqnarray*}
\begin{pmatrix}
A_{s}+A_{t} & -A_{t}\one \\ -\one^{T}A_{t} & \one^{T}A_{t}\one
\end{pmatrix}
\begin{pmatrix}
\bu \\ y
\end{pmatrix}
=
\begin{pmatrix}
\bbf \\ -\one^{T}\bbf
\end{pmatrix} .
\end{eqnarray*}
These graph Laplacians with lower order terms are similar to 
discretized PDE with Dirichlet boundary conditions. 
A solution for this augmented linear system directly results to 
the solution of $A\bu=\bbf$.
\end{itemize}

The present paper focuses on
designing a multilevel preconditioner that is constructed by applying
recursively a space decomposition based on graph matching. 
The aim is to analyze the matching AMLI solver for the graph Laplacian in detail as a first step in gaining 
an in-depth understanding of a multilevel solver elliptic PDEs. 
The extension of the proposed algorithm
to general graph problems is also a subject of current 
research.

\section{Space decomposition based on matching}\label{matching}
In this section, an outline of the basic idea of matching is provided, a commutative 
diagram which can be used to estimate the energy norm of the $\ell_2$ projection
onto the piece-wise constant coarse vector space resulting from a matching in a graph is given, 
and some auxiliary results that are needed later on in the convergence analysis are discussed.

\subsection{Subspaces by graph partitioning and graph matching}

A graph partitioning of $\GG=(\VV,\EE)$ is a set of subgraphs 
$\GG_{i}=(\VV_{i},\EE_{j})$ such that 
\begin{eqnarray*}
\cup_{i}\VV_{i}=\VV
, \quad
\VV_{i}\cap\VV_{j} =\emptyset
, \quad
i\neq j.
\end{eqnarray*}
In this paper, all subgraphs are assumed to be non empty and connected. 
The simplest non trivial example of such a graph partitioning is a matching, i.e, 
 a collection (subset $\MM$) of edges in $\EE$ such that no two
edges in $\MM$ are incident. 

For a given graph partitioning, subspaces of $V=\R^{|\VV|}$ are defined as
\begin{equation*}
V_{c}=\{\bv \in V | \ \bv=\text{ constant on each $\VV_{i}$ } \} .
\end{equation*}
Note that each vertex in $\GG$ corresponds to a connected subgraph $S$ of $\GG$
and every vertex of $\GG$ belongs to exactly one such component. The vectors from
$V_c$ are constants on these connected subgraphs. 
Of importance is the $\ell_2$ orthogonal projection on
$V_c$, which is denoted by $Q$, and defined as follows:
\begin{equation}\label{Q}
(Q \bv)_i = \frac{1}{|\VV_{k}|}\sum_{j\in \VV_{k}} \bv_j 
,\quad
\forall i \in \VV_{k}. 
\end{equation}

Given a graph partitioning, 
the coarse graph $\GG_{c}=\{\VV_{c},\EE_{c}\}$ is defined by 
assuming that all vertices in a subgraph form an equivalence class, 
and that $\VV_{c}$ and $\EE_{c}$ are the quotient set of $\VV$ and $\EE$ 
under this equivalence relation. 
That is, any vertex in $\VV_{c}$ corresponds to a subgraph in 
the partitioning of $\GG$, 
and the edge $(i,j)$ exists in $\EE_{c}$ if and only if 
the $i$-th and $j$-th subgraphs are connected in the graph $\GG$. 
Figure \ref{figure_1} is an example of matching of a graph and 
the resulting coarse graph. 

\begin{figure}[h]
\includegraphics[page=1,scale=0.75]{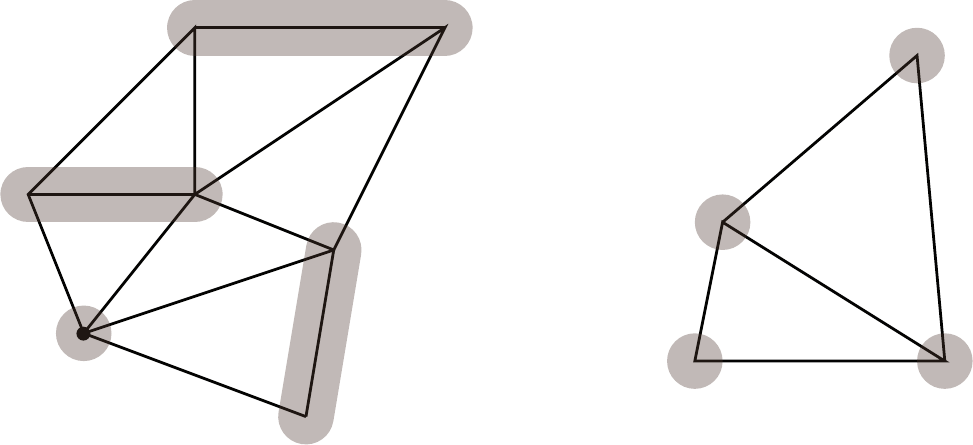}
\caption{Matching $\MM$ on a graph $\GG$ (left) and 
the coarse graph $\GG_c$ (right)}\label{figure_1}
\end{figure}


As mentioned above, the reason to focus on matching is that it simplifies 
the computation of several key quantities used in the upcoming estimates 
derived for a perfect matching and  
it is possible to show that a matching which is not perfect can be analyzed in 
a similar way.

\subsection{Commutative diagram}

Let $B$ be the discrete gradient of a graph Laplacian $A$, as defined in  \eqref{varprob}, 
and $Q$ be defined as in \eqref{Q}.  Assume
that there exists an operator $\Pi_k$ such that 
the following commutative  diagram holds true:
\[
\begin{CD}
{\R^{\abs{\VV}}} @> {B} >> {\R^{\abs{\EE}}} \\
@ V{Q} VV @ VV {\Pi} V \\
{V_c} @>> {B} > {\R^{\abs{\EE}}} \\
\end{CD}
\] 
The proof of this assumption is provided later on.
From the commutative relation $B Q = \Pi B$ it follows that  
\begin{equation}\label{A_norm_of_Q_k_v}
|Q \bv|_A^2  
=
\|BQ \bv\|^2
=
\| \Pi B \bv \|^2 
\le
\|\Pi \|^2 |\bv|_A^2 .
\end{equation}
Thus, 
an estimate on the $A$-semi-norm of $Q$ amounts to 
an estimate of the $\ell_2$ norm of $\Pi$.
In the next subsection, an explicit form of $\Pi$ is constructed and an 
estimate of its $\ell_2$ norm is derived.

\begin{remark}
A more general approach for weighted graph Laplacians 
is to assume that the weight matrix $D\neq I$, 
therefore the bound on the norm $|Q|_{A}$ becomes 
$$
|Q \bv|_A^2  
=
(DBQ \bv,BQ \bv)
=
(D\Pi B \bv,\Pi B \bv)
\le
\|D^{1/2}\Pi_kD^{-1/2}\|^2  |\bv|_A^2 , 
$$
where $D$ can have some negative weights, which results in a matrix $D^{1/2}\Pi_k D^{-1/2}$ this is complex valued. 
A detailed analysis in such a setting and the application of this idea to anisotropic
diffusion problems are discussed in \cite{commutative}. 
\end{remark}

\subsection{Construction of $\Pi$ in case of piece-wise constant spaces}
Here, we proceed with an explicit construction and $\ell_2$ norm estimate of the operator $\Pi$.

For any graph partitioning in which the subgraphs are connected, 
a given edge belongs to the set of ``internal edges'', 
whose vertices belong to the same subgraph, 
or to the set of ``external edges'', 
whose vertices belong to two distinct subgraphs. 
For example, 
let $\GG_{1}$ and $\GG_{2}$ denote the subgraphs 1 and 2 
in Fig. \ref{figure_2}, 
then 
$k_{1}$ is an internal edge and $k_{2}$ is an external edge. 

\begin{figure}[h]
\caption{Connected components and the construction of $\Pi_k$}\label{figure_2}
\hspace{10pt}
\begin{center}
\includegraphics*[page=2,scale=0.75]{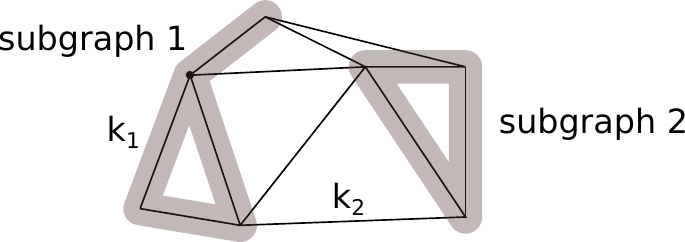}
\end{center}
\end{figure}

Since the vector $Q\bv$ has the same value on the two endpoints of 
the edge $k_{1}$, 
we have that $(BQ \bv)_{k_{1}}=0$. 
Accordingly, all entries in $(\Pi)_{k_{1}}$, 
the $k_{1}$-th row of $\Pi$, are set to zero: 
\begin{eqnarray*}
(\Pi B\bv)_{k_{1}} = (\Pi)_{k_{1}} B\bv = 0 .
\end{eqnarray*}
For the external edge $k_{2}$, 
it follows that $(\Pi)_{k_{2}}$ 
satisfies 
\begin{equation}\label{pi_k_2}
(\Pi)_{k_{2}}(B \bv) 
= 
(BQ \bv)_{k_{2}} = 
\frac{1}{|\VV_{1}|}\sum_{i_1 \in \VV_{1}} \bv_{i_1} - 
\frac{1}{|\VV_{2}|}\sum_{i_2 \in \VV_{2}} \bv_{i_2},
\end{equation}
for every $\bv$.
The following Lemma is useful in computing explicitly the entries of $(\Pi)_{k_{2}}$.
\begin{lemma}
Let $A: \R^{n}\mapsto \R^{n}$ be a positive semidefinite operator 
and let $\{\chi_i\}_{i=1}^n$ be a basis in $\R^{n}$. 
Assume that 
the null space of $A$ is one dimensional, 
namely there exist a nonzero vector $s$ such that 
$\operatorname{Ker}(A)=\operatorname{span}(s)$, 
and 
for every integer $1\le i \le n$ we have $(\chi_i,s)=1$. 
We then have: 
\renewcommand{\labelenumi}{(\roman{enumi})}
\renewcommand{\theenumi}{(\roman{enumi})}
\begin{enumerate}
\item 
\label{discrete-taylor-lemma-1}
For any $i$, the operator $\widetilde{A}:\R^{n}\mapsto \R^{n}$ with  
$\widetilde{A}u=(A u+(\chi_i,u)\chi_i)$ is invertible. 
\item
\label{discrete-taylor-lemma-2}
The following identity holds for all $u\in \R^{n}$:
$$
\frac{1}{(s,s)}(u,s)- (u,\chi_i) = \frac1{(s,s)} (\widetilde{A}^{-1} s,A u).
$$
\end{enumerate}
\renewcommand{\labelenumi}{\arabic{enumi}.}
\renewcommand{\theenumi}{\arabic{enumi}.}
\end{lemma}
\begin{proof}
To establish~\ref{discrete-taylor-lemma-1} it suffices to show that
$\widetilde{A} v=0$ implies $v=0$. 
Assuming that $\widetilde{A} v=0$ for some $v\in \R^{n}$ it follows that:
$$
0=(\widetilde{A}v,v)=(Av,v) + (\chi_i,v)^2.
$$
Note that both terms on the right side of the above identity are
nonnegative and, hence, their sum can be zero if and only if both terms
are zero. Since $A$ is positive semidefinite by assumption with one dimensional
null space, from $(Av,v)=0$ we conclude that $v=\alpha s$ for some
$\alpha\in \R$. 
For the second term we have that 
$0=(\chi_i,v)^2=\alpha^2 (\chi_i,s) ^2$, and since 
$(\chi_i,s)\neq 0$ for all $i$, 
it follows that $\alpha=0$ and hence
$v=0$. This proves~\ref{discrete-taylor-lemma-1}. 

Now, applying \ref{discrete-taylor-lemma-1} the result \ref{discrete-taylor-lemma-2} follows:
\begin{eqnarray*}
\frac1{(s,s)}
(\widetilde{A}^{-1} s,A u) 
& = & 
\frac1{(s,s)}
(\widetilde{A}^{-1} s, Au +(\chi_i,u)\chi_i)
- 
\frac1{(s,s)}
(\widetilde{A}^{-1} s,(\chi_i,u)\chi_i)\\
& = & 
\frac1{(s,s)}
(\widetilde{A}^{-1} s, \widetilde{A}u)
-
\frac1{(s,s)}
(\chi_i,u)(s,\widetilde{A}^{-1} \chi_i)\\
& = & 
\frac1{(s,s)}(s, u)
-
\frac1{(s,s)}
(\chi_i,u)(s,\widetilde{A}^{-1} \chi_i) \\
&=&
\frac1{(s,s)}(u,s) - (u,\chi_{i}) .
\end{eqnarray*}
Here, the equality
$\widetilde A s = As+(\chi_{i},s)\chi_{i} = \chi_{i}$
was used, implying that 
$\widetilde A^{-1} \chi_{i} = s$ .
\end{proof}
\begin{remark}
A special case is given by taking $s=\one$
and $\chi_{i}=e_i$, which denote the standard Euclidean bases.  Then, 
it follows that
\begin{equation}\label{eq:id_spec}
u_i = \langle u \rangle - \frac{1}{n}((A+e_i e_i^T)^{-1}\one,Au)
\end{equation}
in which $\langle u\rangle := {\frac{1}{n}} \sum_{i=1}^n u_i$
denotes the average value of $u$. 
\end{remark}

Next, denote by $B_{m}$ the restriction
of $B$ to a subgraph $\GG_{m}$, 
and set $A_m:=B_m^T B_m$.  Then, 
$\bu_{l}$ can be expressed as 
the $l$-th component of $\bu$ for $\bu$ in \eqref{eq:Taylor}. 
\begin{equation}\label{eq:Taylor}
\bu_l
= 
\langle \bu \rangle_{m} 
+ 
\frac{1}{|\VV_{m}|}
(B_m(A_m + \be_l \be_l^T)^{-1} \one_{m}, B_m \mathcal{I}_{m} \bu) ,
\end{equation}
where $l=1\dots |\VV_{m}|$ are the local indices of the vertex set $\VV_{m}$, 
the operator $\langle \cdot \rangle_{m}$ and the term $\one_{m}$ 
are the averaging operator and the constant vector restricted on the 
subgraph $\GG_{m}$, 
and $\mathcal{I}_{m}: \R^{|\VV|} \mapsto \R^{|\VV_{m}|}$ maps 
the global edge indices to the local edge indices.

Applying this formula for $\GG_1$ and $\GG_2$, 
gives the row of the operator $\Pi$ on the edge $k_{2}$ 
that connects $\GG_1$ and $\GG_2$ as follows 
\begin{equation}\label{row_of_pk_formal}
(\Pi)_{k_{2}} 
= 
C_{1}^{T}\mathcal{I}_{1}^{T} 
+ 
e_{k_{2}}^{T}
-
C_{2}^{T} \mathcal{I}_{2}^{T} . 
\end{equation}
Here $C_{1}$ is given by 
\begin{eqnarray*}
C_{1}
=
\frac{1}{|\VV_{1}|} B_{1}(A_{1} + \be_i \be_i^T)^{-1}\one_{1} 
\end{eqnarray*}
which then makes the summation in \eqref{row_of_pk_formal} valid. 
The vector $C_{2}$ is defined in a similar way. 

Assume that the global indices of the vertices in $\GG_{1}$ and $\GG_{2}$ 
are ordered consecutively as 
decreasing integers starting at $k_{2}-1$ and 
increasing integers starting at $k_{2}+1$. 
Then, 
the $k_{2}$-th row of $\Pi$ can be expressed as 
\begin{equation}\label{row_of_pk}
(\Pi)_{k_{2}} 
= 
\left[
0, \dots, 0, 
C_{1}^{T}, 1, C_{2}^{T}, 
0, \dots, 0
\right]
\end{equation}
where the number $1$ is on the $k_{2}$-th position in this row of $\Pi$. 
Note that from \eqref{eq:Taylor} it follows that
the property \eqref{pi_k_2} holds for 
this construction of $\Pi$, 
since by definition of $Q$
\begin{eqnarray*}
(BQ \bv)_{k_{2}} 
&=& 
\langle \bv \rangle_{1} - \langle \bv \rangle_{2} \\
&=& 
\bv_{i} - \bv_{j} 
+ 
\frac{1}{|\VV_1|}(B_{1}(A_{1} + \be_i \be_i^T)^{-1} \one_{1}, B_{1} \bv_{1}) 
-
\frac{1}{|\VV_2|}(B_{2}(A_{2} + \be_j \be_j^T)^{-1} \one_{2}, B_{2} \bv_{2}) \\
&=&
(e_{k_{2}}, B\bv) + (C_{1},B_{1}\bv) + (C_{2},B_{2}\bv) \\
&=&
(\Pi)_{k_{2}} B\bv,
\end{eqnarray*}
where $k_{2}=(i,j)$, and 
$i$ and $j$, both in local indices, are the incident vertices of $k_{2}$.

\section{A two-level method}\label{twolevel}

In this section, the $\ell_{2}$-orthogonal projection given in 
\eqref{Q} based on a matching $\MM$ is proven to be stable assuming 
the maximum degree of the graph $\GG$ is bounded. 
Then, a two-level preconditioner is derived
and the condition number of the system preconditioned by this two-level method is proven to be
uniformly bounded (under the same assumption). 

\subsection{Two-level stability}

The construction of $\Pi$ for a matching $\MM$ proceeds as follows. 
First, note that all rows of $\Pi$ that correspond to an edge $k=(i,j) \in \MM$ 
are identically zero. 
On the other hand, if the edge $k=(i,j) \notin \MM$, then it is an external edge and, thus,
by \eqref{row_of_pk}, 
the $k$-th row of $\Pi$ is  
\begin{eqnarray*}
(\Pi)_{k}
&=& 
\bigg[
0,\dots,0, \quad
\frac12 (1,-1) 
\left( 
\begin{pmatrix} 1 & -1 \\ -1 & 1 \end{pmatrix} 
+
\begin{pmatrix} 0 \\ 1 \end{pmatrix}^T 
\begin{pmatrix} 0 \\ 1 \end{pmatrix}
\right)^{-1}
\begin{pmatrix} 1 \\ 1 \end{pmatrix}
, \\
&&
1, \quad
-\frac12 (1,-1) 
\left( 
\begin{pmatrix} 1 & -1 \\ -1 & 1 \end{pmatrix} 
+
\begin{pmatrix} 1 \\ 0 \end{pmatrix}^T 
\begin{pmatrix} 1 \\ 0 \end{pmatrix}
\right)^{-1}
\begin{pmatrix} 1 \\ 1 \end{pmatrix}
, \quad 
0,\dots,0 \bigg] \\
&=& 
\big[
0,\dots,0, 
\frac12, 1, -\frac12, 
0,\dots,0
\big] .
\end{eqnarray*}
Hence,
\begin{equation}\label{pk_match_rowwise}
(\Pi)_{kl} 
= 
\left\{
\begin{array}{ll}
1 & k \notin \MM \text{ and } l=k ; \\
\pm \frac12 \qquad & k\not\in \MM, l \in \MM \text{ and } l\cap k\neq \emptyset ; \\
0 & \text{elsewhere} . 
\end{array}\right.
\end{equation}
The alternative way of describing the entries in $\Pi$ is by showing that, 
\begin{equation}\label{pk_match_columnwise}
(\Pi)_{kl} 
= 
\left\{
\begin{array}{ll}
1 & l \notin \MM \text{ and } k=l ; \\
\pm \frac12 \qquad & l \in \MM, k \not\in \MM \text{ and } k \cap l\neq \emptyset ; \\
0 & \text{elsewhere} . 
\end{array}\right. 
\end{equation}

Formula \eqref{pk_match_rowwise} implies that, 
the $k$-th row of $\Pi$ can be a zero row if $k \in \MM$, 
or a row with 3 non-zero entries if $k \notin \MM$, 
which results to 
\begin{eqnarray*}
\|\Pi\|_{\infty} 
=
\max_{k} \sum_{l} |\Pi_{kl}|
=
1+|\pm 1/2|+|\pm 1/2|=2 .
\end{eqnarray*}

Formula \eqref{pk_match_columnwise} implies that, 
the $l$-th column of $\Pi$ can have exactly 1 non-zero entry if $l\notin \MM$, 
or $s$ non-zeros entries whose values are $\pm 1/2$ if $l\in \MM$. 
Here $s$ is the number of edges satisfying 
$k\notin \MM$ and $k\cap l \neq \emptyset$ for 
any given $l\in \MM$, 
thus is bounded by $2d-2$, 
where $d$ is the maximum degree of the graph, 
since an edges can have at most $2d-2$ neighboring edges. 
This leads to 
\begin{eqnarray*}
\|\Pi\|_{1}
=
\max_{l} \sum_{k} |\Pi_{kl}|
=
\max
\big(
1,
(2d-2)|\pm 1/2|
\big)
=
\max
\big(  1, d-1  \big) .
\end{eqnarray*}

On a graph whose maximal degree is larger or equal to 2, 
the estimates on the infinity norm and $\ell_{1}$ norm of $\Pi$ 
result to the following estimate on $\rho(\Pi\Pi^{T})$: 
\begin{eqnarray}\label{rho_pi_pi_t}
\rho(\Pi\Pi^{T}) 
=
\|\Pi\|_{2}^{2}
\leq 
\|\Pi\|_{1}
\|\Pi\|_{\infty}
=2d-2 .
\end{eqnarray}

%

\begin{remark} 
Applying Gerschgorin's theorem directly to the matrix $\Pi\Pi^{T}$ 
leads to a sharper estimate: 
$\rho (\Pi \Pi^{T}) \le d$. 
\end{remark}

%

Formula \eqref{rho_pi_pi_t} implies directly the following lemma. 
\begin{lemma}\label{lem_qa=2}
On any graph whose maximum degree is 2 (e.g. such graph is a path), 
the operator $\Pi$ defined in \eqref{pk_match_rowwise} satisfies 
$\Pi_{k}(B \bv) = (BQ \bv)_k$
and the following estimate holds
$$
|Q|_{A}^{2} \leq \|\Pi\|_{2}^2 \leq 2d-2 = 2 .
$$
\end{lemma}

Numerical tests show that this is a sharp estimate on the semi-norm $|Q|_{A}$ 
and that leads to fast convergent and reliable AMG methods. 


\subsection{A two-level preconditioner}\label{sec:two_lev}
Here, using an estimate of the stability of the matching projection (i.e, the norm $|Q|_{A}$, where
$Q$ is defined via the matching) two-level convergence is established.
Assume that for a graph Laplacian $A: \R^{n} \mapsto \R^{n}$ a perfect matching
is given and consider 
the $n\times n/2$ matrix $P$
whose $k$-th column is given by
\begin{equation}\label{definitionp}
(P)_{k}=\be_{i_{k}}+\be_{j_{k}}, 
\end{equation}
where $k=1,...,n/2$ and $(i_{k},j_{k})$ is the $k$-th edge in $\MM$.
Further, define $Q$ to be the $\ell_{2}$ projection from $\R^{n}$ 
to $\{Pv|v\in\R^{n/2}\}$, i.e.,
$$
Q=P(P^{T}P)^{-1}P^{T} .
$$

Similar to the definition of $P$, define $Y$ as the $n\times n/2$ matrix whose columns
are given by
\begin{equation}\label{definitions}
(Y)_{k}=\be_{i_{k}}-\be_{j_{k}}, 
\end{equation}
where $k=1,...,n/2$ and $(i_{k},j_{k})$ is the $k$-th edge in $\MM$.
Then, the matrix $(Y,P)$ is orthogonal and 
the columns of $Y$ and $P$ form a hierarchical bases, which
can be used to relate the two-level method to a block factorization as follows. 

Given $A$, $P$, and $Y$, define
$$
\widehat A = (Y,P)^{T}A(Y,P)
=\begin{pmatrix}
Y^{T}AY & Y^{T}AP \\
P^{T}AY & P^{T}AP
\end{pmatrix} .
$$
A direct calculation then shows that 
$$
\widehat A
=
L
\begin{pmatrix}
Y^{T}AY & 0 \\
0 & S
\end{pmatrix}
L^{T} ,
$$
where
\begin{equation}\label{schur}
S=P^{T}AP-P^{T}AY(Y^{T}AY)^{-1}Y^{T}AP
\end{equation}
is the Schur complement and 
\begin{equation}\label{definitionl}
L=
\begin{pmatrix}
I & 0 \\
P^{T}AY(Y^{T}AY)^{-1} & I 
\end{pmatrix} .
\end{equation}

Next, define $\GG_{c}$ as the unweighted coarse graph
and denote by $A_{c}$ the graph Laplacian of $\GG_{c}$. 
In contrast to most of the existing AMG methods, 
here 
$A_{c}\neq P^{T}AP$, except in special cases, e.g., for 1 dimensional problems.
Let, $\sigma$ be a positive constant such that
\begin{eqnarray}\label{def-sigma}
\sigma =
\sup_{\bv: (\bv, \one)=0}\frac{(AP\bv, P\bv)}{(A_{c}\bv, \bv)} .
\end{eqnarray}
Then, the fact that all weights in the graph corresponding to $P^{T}AP$ are 
larger than or equal to one implies 
$(AP\bv, P\bv) \geq (A_{c}\bv, \bv), \forall \bv$, 
and 
\begin{eqnarray*}
\frac{(\sigma A_{c}\bv, \bv)}{(AP\bv, P\bv)}\in [1, \sigma] 
,\quad
\forall \bv: (\bv, \one)=0 .
\end{eqnarray*}

Consider the two-level preconditioner $\widehat G$
which uses the coarse graph Laplacian $A_{c}$ by
$$
\widehat G
=
L
\begin{pmatrix}
Y^{T}AY & 0 \\
0 & \sigma A_{c}
\end{pmatrix}
L^{T} .
$$
Let $M$ be a preconditioner for $Y^{T}AY$, and 
$D$ be a preconditioner for the graph Laplacian $A_{c}$.
Then, a two-level preconditioner $\widehat B$ is defined by
\begin{eqnarray}{\label{def:twolevel-b}}
\widehat B
=
\widetilde L
\begin{pmatrix}
M(M+M^{T}-Y^{T}AY)^{-1}M^{T} & 0 \\
0 & \sigma D
\end{pmatrix}
\widetilde L^{T},
\end{eqnarray}
where
$$
\widetilde L
=
\begin{pmatrix}
I & 0 \\
P^{T}AYM^{-1} & I 
\end{pmatrix} .
$$

As observed in \cite{FaVaZiNLA_05} and \cite{panayot-book}, 
this gives a
block matrix representation of the two-level method
\begin{eqnarray*}
I-(Y,P){\widehat G}^{\dagger}(Y,P)^{T} A
&=&
   (I-Y(Y^{T}AY)^{-1}Y^{T}A)
   (I-P(\sigma A_{c})^{\dagger}P^{T}A)
   (I-Y(Y^{T}AY)^{-1}Y^{T}A) \\
I-(Y,P){\widehat B}^{\dagger}(Y,P)^{T} A
&=&
   (I-YM^{-T}Y^{T}A)(I-P(\sigma D)^{\dagger}P^{T}A)(I-YM^{-1}Y^{T}A),
\end{eqnarray*}
where the pseudo-inverse operator denoted by $^{\dagger}$ is used 
since the graph Laplacian is semi-definite. 
The convergence of the two-level method can now be estimated 
by comparing $\widehat A$ and the preconditioner $\widehat B$. 
 
The remainder of this section is dedicated to establishing a spectral equivalence 
between $\widehat A$ and $\widehat B$ for the two-level matching algorithm. 
The proof uses the following Lemma. 

\begin{lemma}\label{lem_inf}
For any $\bx \in \Re^{n/2}$ the Schur complement $S$ 
as given in (\ref{schur}) satisfies 
\begin{equation}
(S \bx, \bx)
=
\inf_{\bw}\big( A(Y \bw+P \bx) , (Y \bw+P \bx) \big).
\end{equation}
\end{lemma}
\begin{proof}
Note that
\begin{eqnarray*}
\big( A Y (Y^{T}AY)^{-1} Y^{T} A P \bx, P\bx \big)
&=&
\big( A Y (Y^{T}AY)^{-1} Y^{T} A P \bx, Y (Y^{T}AY)^{-1} Y^{T} A P \bx \big) \\ 
&=&
\|Y (Y^{T}AY)^{-1} Y^{T} A P \bx\|_{A}^{2}, 
\end{eqnarray*}
because here, $Y(Y^{T}AY)^{-1}Y^{T}AP\bx$ is the $A$ orthogonal projection of $P\bx$ onto the space 
spanned by the columns of $Y$ and, thus, minimizes the distance (in $A$ norm) between $P\bx$ and this
space. Hence,
\begin{align*}
( S \bx , \bx ) 
&=
\| P \bx \|_{A}^{2} - \|Y(Y^{T}AY)^{-1}Y^{T}AP\bx\|_{A}^{2} \\
&=
\inf_{\bw} \big( A(Y\bw+P\bx) , (Y\bw+P\bx) \big) 
\qedhere
\end{align*}
\end{proof}
\noindent
Let $\widehat\one$ be a vector satisfying 
$(Y,P)\widehat\one = \one$, 
then the following lemma now holds.
\begin{lemma}\label{lem_Ghat}
Let $c_{g}=\sigma |Q|_{A}^{2}$, 
where $\sigma$ is defined as in \eqref{def-sigma}. 
Then for any $\bv$, such that $(\bv, \widehat\one) = 0$, we have
\begin{eqnarray} \label{est_Ghat}
\frac{( \widehat G \bv, \bv )}{(\widehat A \bv, \bv )}
\in [1, c_{g}]  .
\end{eqnarray}
\end{lemma}
\begin{proof}
By Lemma~\ref{lem_inf} we have 
\begin{eqnarray*}
( A P \bx, P\bx) 
\geq 
\inf_{\bw}\big( A (Y\bw+P\bx), (Y\bw+P\bx) \big) .
\end{eqnarray*}
Furthermore,
\begin{eqnarray*}
\frac{(AP\bx, P\bx)}{(S \bx, \bx)}
&=&
\frac{(AP\bx, P\bx)}{\inf_{\bw} \big( A (Y\bw+P\bx), (Y\bw+P\bx) \big) } \\
&=&
\sup_{\bw}\frac{(AP\bx, P\bx)}{\big( A (Y\bw+P\bx), (Y\bw+P\bx) \big)} \\
&=&
\sup_{\bu=Y\bw+P\bx}
\frac{(AQ\bu, Q\bu)}{(A\bu, \bu)} 
\leq
\sup_{\bu}
\frac{(AQ\bu, Q\bu)}{(A\bu, \bu)} 
=
|Q|_{A}^{2} .
\end{eqnarray*}
Note that the only difference between the preconditioners $\widehat G$ and $\widehat A$ 
is that the former matrix uses $\sigma A_{c}$, whereas the latter uses $S$ to define the 2-2 block. 
The spectral equivalence constant between the
operators $\sigma A_{c}$ and $S$ 
is obtained as follows:
\begin{eqnarray*}
\inf_{\bu}
\frac{\sigma (A_{c} \bu, \bu)}{(AP \bu, P\bu)} 
\inf_{\bv}
\frac{(AP \bv, P \bv)}{(S \bv, \bv)} 
\leq
\frac{\sigma (A_{c} \bw, \bw)}{(S \bw, \bw)} 
\leq
\sup_{\bu}
\frac{\sigma (A_{c} \bu, \bu)}{(AP \bu, P \bu)} 
\sup_{\bv}
\frac{(AP \bv, P\bv)}{(S \bv, \bv)} , \\
\quad 
\forall \bw:(\bw,\bm{1})=0, 
\end{eqnarray*}
which implies
\begin{eqnarray*}
\frac{\sigma (A_{c} \bw, \bw)}{(S \bw, \bw)} 
\in 
[1,\sigma |Q|_{A}^{2}]
,\qquad 
\forall \bw:(\bw,\bm{1})=0 .
\end{eqnarray*}
Hence, for any $\bx$ and $\by$
$$
\frac{
\begin{pmatrix}
\bx \\ \by 
\end{pmatrix}^T
\begin{pmatrix}
Y^{T}AY & 0 \\
0 & \sigma A_{c}
\end{pmatrix}
\begin{pmatrix}
\bx \\ \by 
\end{pmatrix}
}{
\begin{pmatrix}
\bx \\ \by 
\end{pmatrix}^T
\begin{pmatrix}
Y^{T}AY & 0 \\
0 & S
\end{pmatrix}
\begin{pmatrix}
\bx \\ \by
\end{pmatrix}
}\\
=
\frac{(AY\bx, Y\bx) + (AP\by, P\by)}{(AY\bx, Y\bx) + (S\by, \by)} \\
\in [1, \sigma |Q|_{A}^{2}] ,
$$ 
which is equivalent to 
\eqref{est_Ghat}
since $L$ is nonsingular. 
\end{proof}

Since the two-level method $\widehat G$ requires 
exact solvers for $Y^{T}AY$ and the graph Laplacian $A_{c}$, 
the convergence rate of a method that uses $\widehat B$ which is defined by replacing
these exact solves with approximate ones is of interest.
Combining Lemma~\ref{lem_Ghat} and the two-level convergence estimate 
(Theorem 4.2 in \cite{FaVaZiNLA_05}), 
yields the following result.
\begin{theorem}\label{thm_Bhat}
If the preconditioners $M$ and $D$ are spectrally equivalent to 
$Y^{T}AY$ and $A_{c}$ such that 
$$
\frac
{\big( (M^{T}+M-Y^{T}AY)^{-1}M \bu, M\bu \big)}{(A Y\bu, Y\bu)} 
\in 
[1, \kappa_{s}]
\quad \mbox{and} \quad
\frac{(D\bw, \bw)}{(A_{c}\bw, \bw)}
\in
[1, \eta]
,\quad \forall \bu, \bw:(\bw,\one)=0, 
$$ 
then 
\begin{equation}\label{est_Bhat}
\frac{(\widehat B \bv, \bv)}{(\widehat A \bv, \bv)}
\in [1, (\kappa_{s}+\sigma\eta-1)|Q|_{A}^{2}] 
, \qquad
\forall \bv:(\bv,\widehat\one) = 0 .
\end{equation} 
\end{theorem}
\noindent Note that this estimate reduces to \eqref{est_Ghat} when 
$M=Y^{T}AY$ and $D=A_{c}$.

\begin{subsection}{Convergence estimate for matching}

We here show the sharpness of the estimation given by Theorem \ref{thm_Bhat} 
when the graph Laplacian corresponds to a structured grid, 
and the coarse space is given by aligned matching. 

Define an $m$-dimensional hypercubic grid as 
the graph Laplacian $\GG=(\VV,\EE)$ such that 
the following conditions are satisfied. 
\begin{enumerate}
\item
A vertex $i_{v}\in \VV$ corresponds to an vector $v\in \R^{m}$, 
and $(v,e_{j})\in [1, 2, \dots, s_{j}]$, $j=1,2,\dots, m$. 
Here $e_{j}$ is an Euclidean basis and 
$s_{1},s_{2},\dots, s_{m}$ are given positive integers that represent 
the numbers of vertices along all dimensions. 
\item
An edge $k=(i_{u}, i_{v})$ is in the edge set $\EE$ if and only 
$u-v=e_{j}$ and $j\in [1,2,\dots, m]$. 
\end{enumerate}

Then the energy norm $|Q|_{A}$ can be estimated for aligned matching on 
a hypercubic grid $A$. 
\begin{lemma}\label{equi}
Let $\GG$ be an $m$-dimensional hypercibic grid and 
$k\in [1,2,\dots,m]$ is a fixed dimension. 
Assume that $s_{k}$ is an even number. 
The matching along the $k$-th dimension is defined as 
\begin{eqnarray*}
\MM=
\{
l=(i_{v},i_{v+e_{k}})
|
v\in \VV, 
\text{ and $(v,e_{k})$ is an odd number }
\} .
\end{eqnarray*}
Let $Q$ be the $\ell_{2}$ projection onto the 
piecewise constant space resulting from the matching $\MM$. 
Then $Q$ satisfies 
$|Q|_{A}\leq 2$. 
\end{lemma}
\begin{proof}
Define the set $\Omega$ be the collection of all edges along the $k$-th dimension, 
as 
\begin{eqnarray*}
\Omega
=
\{
l=(i_{u},i_{v})
|
v-u=e_{k}
\} .
\end{eqnarray*}
Also define $\overline\Omega=\EE\setminus\Omega$ and 
the graph Laplacians $A_{\Omega}$ and $A_{\overline\Omega}$, 
derived from $\Omega$ and $\overline\Omega$ respectively. 
\end{proof}

The graphs in the set $\Omega$ are paths, whose maximum degree is 2, 
and $\MM\subset\Omega$ is a also matching on 
these paths. 
Therefore by Lemma \ref{lem_qa=2} it is true that
\begin{eqnarray}\label{q-a-omega}
(A_{\Omega} Q\bu,Q\bu) \leq 2 (A_{\Omega}\bu, \bu) .
\end{eqnarray}
\noindent
On the other hand, 
the matching is aligned on the set $\overline\Omega$, 
meaning that any two matched pairs are connected 
through 0 or 2 edges in $\overline\Omega$, 
thus the edges in set $\overline\Omega$ can then be subdivided into 
many sets of edges of the same type, 
one of which is shown
in Fig. \ref{figure_4}.
\begin{figure}[h]
\includegraphics[page=4,scale=0.75]{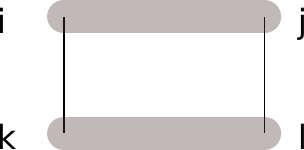}
\caption{Matching $\MM$ on a subset of $\overline\Omega$}\label{figure_4}
\end{figure}
Notice that in this figure, the edge $(i,k)$ and $(j,l)$ are in $\overline\Omega$, 
while $(i,j)$ and $(k,l)$ are in $\MM$. 
Using the definition of $Q$, 
the energy norm of $Q$ is estimated on the the subset of $\overline\Omega$ 
indicated by Fig. \ref{figure_4}, by 
\begin{eqnarray*}
2 \left(\frac{u_{i}+u_{j}}{2} - \frac{u_{k}+u_{l}}{2}\right)^{2}
&=&
\frac12 \left((u_{i}-u_{k}) + (u_{j}-u_{l})\right)^{2}\\
&\leq&
(u_{i}-u_{k})^{2} + (u_{j}-u_{l})^{2}.
\end{eqnarray*}
Thus implies that 
\begin{eqnarray}\label{q-a-omega-bar}
(A_{\overline\Omega}Q\bu,Q\bu) \leq (A_{\overline\Omega}\bu,\bu) .
\end{eqnarray}
Combining \eqref{q-a-omega} and \eqref{q-a-omega-bar} results that 
$(AQ\bu,Q\bu)\leq 2 (A\bu, \bu)$, 
or $|Q|_{A}\leq 2$. 


\begin{remark}\label{remark-line-aggregation}
A similar estimate follows for aligned partitionings consisting of 
line segments of size $m$.  
Namely, in this case it can be shown that $|Q|_{A}^{2}\leq m$ holds. 
This estimate in turn agrees with properties of Chebyshev polynomials, 
suggesting the use of an AMLI method equipped with 
certain Chebyshev polynomials. 
Comparing this result with the result from Theorem \ref{equi} suggests that using a more shape regular partitioning 
rather than one consisting of lines is more appropriate since this gives smaller values of the semi-norm $|Q|_{A}$. 
\end{remark}

A bound on the constant $\kappa_{s}$  follows by using that $Y^{T}AY$ is well conditioned 
and that its condition number depends on the degree of the graph, 
but not on the size of the graph. 
\begin{lemma}
Let $\MM$ be the perfect matching on a graph maximum whose degree is $d$, 
and let $S$ be defined as in \eqref{definitions}, 
then we have
\begin{eqnarray*}
\frac{(AY \bw, Y\bw)}{(\bw, \bw)} 
\in 
[4,2d] 
,\quad \forall \bw\neq 0.
\end{eqnarray*}
\end{lemma}
\begin{proof}
The $A$-norm of the vector $Y\bw$ is computed by definition:
\begin{eqnarray*}
(AY\bw, Y\bw)
\geq 
\sum_{k=(i,j)\in\MM} \big((Y\bw)_{i}-(Y\bw)_{j}\big)^{2}
=
\sum_{k=(i,j)\in\MM} \big((Y\bw)_{i} + (Y\bw)_{i}\big)^{2}
=
4\bw^{T}\bw .
\end{eqnarray*}
We also have 
\begin{equation*}
\rho(Y^{T}AY) 
\leq 
\|Y^{T}AY\|_{1} 
\leq 
\|Y^{T}\|_{1} \|A\|_{\infty} \|Y\|_{\infty}
=2d .
\qedhere
\end{equation*}
\end{proof}
From the Lemma it follows that for any $ \epsilon >0$ there exists 
a smoother $M$ such that 
the bound on the constant $\kappa_{s}$ in Theorem \ref{thm_Bhat} is 
\begin{eqnarray*}
\kappa_{s} \leq 1+\epsilon .
\end{eqnarray*}
This result in turn implies that 
an efficient solver for $Y^{T}AY$ can be constructed by applying 
a few Conjugate Gradient iterations with an overall cost that is linear with respect to 
the size of $Y^{T}AY$. 

The constant $\sigma$ in \eqref{def-sigma} can be estimated by checking the weights of the 
graph for the graph Laplacian $P^{T}AP$. 
Taking any two distinct subgraphs (edges) in the matching, say the $k$-th and $l$-th such that $k\neq l$, 
it follows that the corresponding entry $(P^{T}AP)_{kl}$ is equal to the number of exterior edges that connect to 
these subgraphs. 
For an aligned matching aligned a fixed dimension of a hypercubic grid, 
these weights are bounded by $2$. 
For any general graph $A$, the weights in $P^{T}AP$ are bounded by $4$, 
since there are at most $4$ distinct edges that connect to any other 2 distinct edges. 
Then, letting $A_{c}$ to denote the unweighted graph Laplacian 
on the graph defined by $P^{T}AP$, 
and noting that all off-diagonal entries of $A_c$ are equal to $-1$, 
it follows that 
$$\sigma=
\left\{
\begin{array}{ll}
2 \qquad & \text{for an aligned matching on a hypercubic grid of any dimension; } \\
4 & \text{for a given matching on any graph. }
\end{array}
\right.
$$
\begin{remark}
These estimates 
can be generalized to
other subgraph partitionings in a similar way. 
As an example, consider 
again a graph for a hypercubic grid of any dimension.  Then, 
for line aggregates of size $m$ (aligned with the grid)
the following estimate holds
\begin{eqnarray*}
|Q|_{A}^{2}\leq m
,\quad
\kappa_{s}\leq 1+\epsilon
,\quad
\eta=1
,\quad
\sigma\leq m .
\end{eqnarray*} 
Such estimates give insight into the design of a nearly optimal
multilevel method.  Moreover, the bounds are sharp enough, namely, the
corresponding multilevel method can be proven to have convergence rate
$\approx (1-1/\log n)$ and $O(n\log n)$ complexity.
\end{remark}

\end{subsection}

\section{Algebraic multilevel iteration (AMLI) based on matching}\label{amli}




In this section, a multilevel method that uses recursively the
two-level matching methods from Section~\ref{sec:two_lev} in combination with a
polynomial stabilization, also known as Algebraic Multilevel Iteration (AMLI) cycle is analyzed. 
Here, the focus is on proving a nearly optimal convergence rate, 
that is, one which is nearly independent of the number of unknowns $n$, 
and at the same time has low computational complexity.   

\subsection{Multilevel hierarchy}

Assume that $A_{J}=A$ is an $n\times n$ graph Laplacian matrix 
where $n=2^{J}$. 
For $k=1,\ldots,J$ define 
the matching $\MM_k$ and the prolongation operator $P_{k}$ 
according to \eqref{definitionp}, 
then compute the graph Laplacian $A_{k}$ of the coarse graph $\GG_{k}$ 
(Recall that, $A_{k-1}\neq P_{k}^{T}A_{k}P_{k}$). 
The index $k$ starts at 1 because 
the analysis is simpler if  the coarsest graph has more than 1 vertex. 
Also, define $Y_{k}$ and $L_{k}$ for $A_{k}$ as in \eqref{definitions} and 
\eqref{definitionl},
and let the two-level preconditioner $\widehat G_{k}$ on each level $k$ be given by 
\begin{eqnarray*}
\widehat G_{k}
=
L_{k}
\begin{pmatrix}
Y_{k}^{T}A_{k}Y_{k} & 0 \\
0 & \sigma A_{k-1} 
\end{pmatrix}
L_{k}^{T}
, \qquad k=2,\dots,J .
\end{eqnarray*}
Then an AMLI preconditioner is defined recursively by
\begin{eqnarray*}
B_{1}^{-1} &=& A_{1}^{\dagger}, \\
\widehat B_{k}^{-1} 
&=&
L_{k}^{-T}
\begin{pmatrix}
(Y_{k}^{T}A_{k}Y_{k})^{-1} & 0 \\
0 & \sigma^{-1} B_{k-1}^{-1}q_{k-1}(A_{k-1}B_{k-1}^{-1})
\end{pmatrix}
L_{k}^{-1}, 
\quad k=2,\dots,J ,  \\
B_{k}^{-1}
&=&
(Y_{k},P_{k})^{T}\widehat B_{k}^{-1} (Y_{k},P_{k}), 
\qquad k=2,\ldots,J ,
\end{eqnarray*}
where $q_{k}(t)$ is a polynomial on that determines 
a special coarse level correction on the $k$-th level. 
In this case, where an AMLI W-cycle is used, $q_{k}(t)$ is a linear function 
for all $k$. 

In the remainder of this section, sufficient conditions for guaranteeing the 
spectral equivalence between the multilevel preconditioner $B_{J}$, as defined
above, and the graph Laplacian $A$ are derived. We first prove two
auxiliary results, which are needed in the analysis below. 
\begin{proposition}\label{prop:spdag}
  Let $A: V\mapsto V$ and $G:V\mapsto V$ be symmetric positive
  semidefinite operators on a finite dimensional real Hilbert space $V$. Suppose
  that the following spectral equivalence holds:
\begin{equation}\label{sp}
c_0 (A\bv,\bv) \le (G \bv, \bv) \le c_1 (A \bv,\bv), \quad c_0>0,\quad c_1 > 0. 
\end{equation}
Then, we also have that
\begin{equation}\label{spdag}
c_1^{-1}(A^\dagger \bv,\bv) \le (G^\dagger  \bv, \bv) \le c_0^{-1} (A^\dagger \bv,\bv). 
\end{equation}
\end{proposition}
\begin{proof}
Observe that the spectral equivalence given in~\eqref{sp} implies that 
$A$ and $G$ have the same null-space (and also same range, because they are
symmetric). Also, note that, if $\bv$ is in this null
space, then~\eqref{spdag} trivially holds. Thus, without loss of
generality, we restrict  our considerations below to $\bv$ from the range of $G$ and $A$. 

After change of variables $\bw=\big( A^{\dagger} \big)^{1/2}\bv$
from the upper bound in \eqref{sp} we may conclude that 
\[
\frac{\| G^{1/2}\big(A^\dagger\big)^{1/2}\bw \|^2}{\|\bw\|^2}\le c_1, 
\quad\mbox{and hence,}\quad
\| G^{1/2}\big(A^\dagger\big)^{1/2}\|^2\le c_1. 
\]
Since $ G^{1/2}\big(A^\dagger\big)^{1/2} =
\left( \big( A^\dagger \big)^{1/2}G^{1/2}\right)^T$, we obtain that  
$\|G^{1/2} \big( A^\dagger \big)^{1/2}\| = \| \big(A^\dagger\big)^{1/2} G^{1/2}\|$. 
Using this identity, the estimate above, we have for all and all $\bu$
and all $\bw=[G^\dagger]^{1/2}\bu$:
\[
c_1\ge 
\| \big( A^\dagger \big)^{1/2} G^{1/2}\|^2\ge 
\frac{\| \big( A^\dagger \big)^{1/2} G^{1/2}\bw\|^2}{\|\bw\|^2},
\quad\mbox{and hence,}\quad
c_1\ge 
\frac{\| \big( A^\dagger \big)^{1/2} \bu\|^2}{\| \big( G^{\dagger} \big)^{1/2}\bu\|^2}.
\]
The estimate given above clearly implies that 
$c_1^{-1}(A^\dagger \bu,\bu) \le (G^\dagger \bu, \bu)$, and this is the lower bound
in~\eqref{spdag}.  The upper bound in~\eqref{spdag} follows by
interchanging the roles of $G$ and $A$ and basically repeating the
same argument.
\end{proof}
The elementary results in the next proposition are used later in the
proof of Lemma~\ref{lem_BkAk}.
\renewcommand{\labelenumi}{(\roman{enumi})}
\begin{proposition}\label{prop:qtheta}
  Let  $\theta\in[0,1]$ and define
  $q(t;\theta)=\dfrac4{\theta+1}(1-\dfrac{t}{\theta+1})$ and
  $\widetilde{q}(t;\theta)=tq(t;\theta)$. Then,
\begin{enumerate}
\parskip=3pt
\item $\displaystyle\max_{t\in [\theta,1]} \widetilde{q}(t;\theta) =1$;
\item $\displaystyle\min_{t\in [\theta,1]} \widetilde{q}(t;\theta) =
  \widetilde{q}(\theta;\theta)= \widetilde{q}(1;\theta)$ ;
\item  $\dfrac{d\widetilde{q}(1;\theta)}{d\theta} \ge 0$ (monotonicity). 
\end{enumerate}
\end{proposition}
\begin{proof}
  The proof of  (i) and (ii) follow from the identity
  $\widetilde{q}(t;\theta)=1-\big( 2t/(\theta+1)-1 \big)^2$. 
The proof of (iii) is also straightforward and follows from the fact
that $\theta\in [0,1]$ and hence
\[
\frac{d\widetilde{q}(1;\theta)}{d\theta} 
=
\frac4{(\theta+1)^{2}} \left( \frac2{\theta+1}-1 \right) \ge 0.
\qedhere
\]
\end{proof}
Next we derive estimates for the growth of the terms in a 
sequence, recursively defined using $\widetilde{q}(1;\theta)$, 
which we use later to bound the convergence rate. 
\begin{proposition}\label{prop:sequence}
  Let, $1\le c \le 4$ be a given
  constant, and  $q(t;\theta)=\dfrac4{\theta+1}(1-\dfrac{t}{\theta+1})$ and
  $\widetilde{q}(t;\theta)=tq(t;\theta)$ (as in Proposition~\ref{prop:qtheta}).
Define, 
\begin{equation}\label{eq:sequence}
\theta_1=1; \qquad \theta_{k+1} = \frac{1}{c}\widetilde{q}(1;\theta_k),
\quad\mbox{for}\quad
k=1,2,\ldots
\end{equation}
Then, the following are true for $k=1,2,\ldots$:
\begin{enumerate}
\parskip=3pt
\item $\dfrac{2}{\sqrt{c}}-1 \le \theta_{k+1}\le\theta_k\le 1$ ; 
\item $\theta_k \ge  \max \left\{\dfrac2{\sqrt{c}}-1,\dfrac{1}{2k-1+\log k}\right\}$. 
\end{enumerate}
\end{proposition}
\begin{proof}
The first item (i) follows from algebraic manipulations and the
estimates given in Proposition~\ref{prop:qtheta}. 
To show that $\theta_{k+1} \leq \theta_{k}$, we assume 
that $\theta_{k}\geq 2/\sqrt{c}-1$ (which is certainly true for
$k=1$. To prove that $\theta_{k+1}\geq 2/\sqrt{c}-1$ we observer that
from $\theta_{k}\geq 2/\sqrt{c}-1$, 
the monotonicity property in Proposition~\ref{prop:qtheta}
item (iii), implies that
\[
\theta_{k+1} = \frac1c \tilde{q}(1;\theta_{k}) 
\geq 
\frac1c \tilde{q}\left(1;\frac2{\sqrt{c}}-1 \right) = \frac2{\sqrt{c}}-1 .
\]
Using again that $\theta_{k}\geq 2/\sqrt{c}-1$ gives aso that
\[
\theta_{k+1}-\theta_{k}=
\frac{\theta_{k}}{(\theta_{k}+1)^{2}}
\left(\frac4c-(\theta_{k}+1)^{2}\right)\leq 0 .
\]

The proof of the second item (ii) is a bit more involved. We prove this
item by deriving an upper bound on 
$\zeta_k=\frac{1}{\theta_k}$. 
Observe that, from the
recurrence relation for $\theta_k$ we have 
\begin{eqnarray}\label{recursive-formula-zeta-k}
\zeta_{k+1}= 
\frac c4(\zeta_{k}+2+\frac1{\zeta_{k}})
,\qquad
\zeta_{1}=1. 
\end{eqnarray}
We first show that the faster growing sequence above is for $c=4$. 
Indeed, let 
\begin{equation*}
s_{k+1} = s_{k}+2+\frac1{s_{k}}
,\qquad 
s_{1}=1.
\end{equation*}
A standard induction argument shows that 
\[
\zeta_{k} \leq s_{k} 
,\qquad \text{ and } \qquad
2k-1 \leq s_{k}
,\qquad 
\forall k .
\]
Expand $s_{k}$ by the recursive formula and we have 
\[
s_{k}
=
s_{1}+2(k-1)+\sum_{i=1}^{k-1}\frac1{s_{i}}
\leq
1+2(k-1)+\sum_{i=1}^{k-1}\frac1{2i-1}
\leq
2k+\ln k +1 ,
\]
which provides an upper bound of $\zeta_{k}$, and 
hence $1/(2k+\ln k-1)$ is a lower bound of $\theta_{k}$. 
\end{proof}

The following Lemma provides a spectral equivalence relation
between $\widehat G_{k}^{\dagger}$ and $\widehat B_{k}^{-1}$.
\begin{lemma}\label{lem_GkBk}
If  $\lambda_{1} \leq \lambda(B_{k}^{-1} A_{k}) \leq \lambda_{2}$
and $t q_{k}(t)>0$ for $\lambda_{1}\leq t \leq \lambda_{2}$, then
\begin{eqnarray} \label{b-1/m-1}
\min\{1, \min_{\lambda_{1}\leq t \leq \lambda_{2}}t q_{k}(t)\}
\leq
\frac{(\widehat B_{k+1}^{-1}\bv, \bv)} { ( \widehat G_{k+1}^{\dagger}\bv, \bv ) }
\leq
\max\{1, \max_{\lambda_{1}\leq t \leq \lambda_{2}}t q_{k}(t)\}, \qquad\qquad \\
\forall \bv:(\bv,\widehat\one)=0,\quad
k=1,\dots,J-1 . \nonumber
\end{eqnarray}
\end{lemma}
\begin{proof}
For any vector $\bv$, 
$$
\frac {\big( q_{k}(A_{k}B_{k}^{-1})\bv, B_{k}^{-1}\bv \big)}{(A_{k}^{\dagger}\bv, \bv)}
=
\frac{\big( 
q_{k}(A_{k}^{\frac12}B_{k}^{-1}A_{k}^{\frac12})(A_{k}^{\frac12})^{\dagger}\bv, 
A_{k}^{\frac12}B_{k}^{-1}A_{k}^{\frac12}(A_{k}^{\frac12})^{\dagger}\bv \big)}{(A_{k}^{\dagger}\bv, \bv)} \\
=
\frac{\big( q_{k}(Z)\bw, Z\bw \big)}{(\bw, \bw)},
$$
where $\bw=(A_{k}^{\frac12})^{\dagger}\bv$ 
and $Z=A_{k}^{\frac12}B_{k}^{-1}A_{k}^{\frac12}$.
Further, since $Z$ has the same eigenvalues as $B_{k}^{-1} A_{k}$,  we
conclude that 
$$
\min_{\lambda_{1}\leq t \leq \lambda_{2}}t q_{k}(t)
\leq 
\frac {\big( q_{k}(A_{k}B_{k}^{-1})\bv, B_{k}^{-1}\bv \big)}{(A_{k}^{\dagger}\bv, \bv)}
\leq 
\max_{\lambda_{1}\leq t \leq \lambda_{2}}t q_{k}(t) .
$$
This implies that for any $\bx$ and $\by$, 
\begin{eqnarray*}
&&
\frac {
\begin{pmatrix}
\bx \\ \by
\end{pmatrix}^T
\begin{pmatrix}
(Y_{k+1}^{T}A_{k+1}Y_{k+1})^{-1} & 0 \\
0 & \sigma^{-1} B_{k}^{-1}q_{k}(A_{k}B_{k}^{-1})
\end{pmatrix}
\begin{pmatrix}
\bx \\ \by
\end{pmatrix} 
}{
\begin{pmatrix}
\bx \\ \by
\end{pmatrix}^T
\begin{pmatrix}
(Y_{k+1}^{T}A_{k+1}Y_{k+1})^{-1} & 0 \\
0 & \sigma^{-1} A_{k}^{-1}
\end{pmatrix}
\begin{pmatrix}
\bx \\ \by
\end{pmatrix}
} \\
&=&
\frac
{\big( (Y_{k+1}^{T}A_{k+1}Y_{k+1})^{-1}\bx, \bx \big)
+
\sigma^{-1}(B_{k}^{-1}q(A_{k}B_{k}^{-1})\by, \by)}
{\big( (Y_{k+1}^{T}A_{k+1}Y_{k+1})^{-1}\bx, \bx \big)
+
\sigma^{-1}(A_{k}^{-1}\by, \by)} \\
&\in& 
\Big[
\min\{1, \min_{\lambda_{1}\leq t \leq \lambda_{2}}t q(t)\}, 
\max\{1, \max_{\lambda_{1}\leq t \leq \lambda_{2}}t q(t)\}
\Big] ,
\end{eqnarray*}
and, hence, by using the definition of $\widehat G_{k}$ and $\widehat B_{k}^{-1}$,
it follows that 
\begin{equation}\label{rel_GkBk}
\frac {(\widehat B_{k+1}^{-1}\bv, \bv)}{(\widehat G_{k+1}^{\dagger}\bv, \bv)}
\in
\Big[
\min\{1, \min_{\lambda_{1}\leq t \leq \lambda_{2}}t q(t)\}, 
\max\{1, \max_{\lambda_{1}\leq t \leq \lambda_{2}} t q(t)\}
\Big] .
\qedhere
\end{equation}
\end{proof}

Combining the above lemma with Theorem (\ref{est_Ghat}) 
the spectral equivalence between $B_{k}^{-1}$ and $A_{k}^{\dagger}$,
$k=1,\ldots,J$  follows and is shown in the next Lemma.  

\begin{lemma}\label{lem_BkAk}
  Assume that the two level preconditioner $G_k$ satisfies
\begin{equation}\label{est_AkGk}
(\widehat A_{k}\bv, \bv)
\leq (\widehat G_{k}\bv, \bv) 
\leq c_{g} (\widehat A_{k}\bv, \bv)
, \quad 
\forall \bv \mbox{ and } k=2,\ldots,J.
\end{equation}
with constant $c_{g}$, such that $1\leq c_{g}\le 4$. 
Define 
\begin{equation}\label{tq(t)}
q_{k}(t)=q(t,\theta_k), 
\end{equation}
where $\theta_k$ are defined as 
\begin{equation*}
\theta_1=1; \qquad \theta_{k+1} = 
\frac{1}{c_g}\widetilde{q}(1;\theta_k)=\frac{t}{c_g}q_k(1).
\end{equation*}
Then, the following inequalities hold for all $\bv:(\bv,\one)=0 \text{ and } k=1,\dots,J$.
\begin{eqnarray} 
&&\theta_k
\leq 
\frac{(B_{k}^{-1}\bv, \bv)}{(A_{k}^{\dagger}\bv, \bv)}
\leq 
1, \label{m-1/a-1}\\
&& \max \left\{\frac2{\sqrt{c}}-1,\frac{1}{2k+\ln k+1}\right \}
\leq \frac{(B_{k}^{-1}\bv, \bv)}{(A_{k}^{\dagger}\bv, \bv)}.
\label{m-1/a-20}
\end{eqnarray}
\end{lemma}
\begin{proof}
  We give a proof of \eqref{m-1/a-1} by induction. Clearly, for $k=1$,
  $B^{-1}_1=A_1^\dagger$, and hence, \eqref{m-1/a-1} holds. We assume
  that the inequalities \eqref{m-1/a-1} hold for $k=l$ and we aim to
  prove them for $k=l+1$. For all $\bv$ such that $(\bv,\one)=0$ we
  have
$$
\frac {(\widehat B_{l+1}^{-1}\bv, \bv)}{(\widehat
  A_{l+1}^{\dagger}\bv, \bv)} = 
\frac {(\widehat G^\dagger_{l+1}\bv, \bv)}{(\widehat A^\dagger_{l+1}\bv, \bv)}
\frac {(\widehat B_{l+1}^{-1}\bv, \bv)}{(\widehat G_{l+1}^{\dagger}\bv, \bv)}
$$
Then, from \eqref{est_AkGk}, Proposition~\ref{prop:spdag} and
Proposition~\ref{lem_GkBk} (applied in that order) it follows that    
\[
\frac1{c_g}\le 
\frac{(\widehat G^\dagger_{l+1}\bv, \bv)}{(\widehat
  A^\dagger_{l+1}\bv, \bv)}
\le 1,\quad\mbox{and}\quad
\min\{1, \min_{ t\in[\theta_k,1]}t q_{k}(t)\}\le
\frac{(\widehat B_{l+1}^{-1}\bv, \bv)}{(\widehat G_{l+1}^{\dagger}\bv,
  \bv)}
\le
\max\{1, \max_{ t\in[\theta_k,1]}t q_{k}(t)\}. 
\]
Next, by Proposition~\ref{prop:qtheta} and Proposition~\ref{prop:sequence} we find that
\begin{equation*}
\theta_{l+1}=\frac{1}{c_g}\min\{1, \min_{ t\in[\theta_k,1]}t q_{l}(t)\}\le
\frac{(\widehat B_{l+1}^{-1}\bv, \bv)}{(\widehat A_{l+1}^{\dagger}\bv, \bv)}
\le
\max\{1, \max_{ t\in[\theta_k,1]}t q_l(t)\}=1. 
\end{equation*}

Finally, from the definition of $B_{k}^{-1}$ and $A_{k}^{-1}$ in
terms of $\widehat B_{k}^{-1}$ and $\widehat A_{k}^{\dagger}$, it
immediately follows that
\begin{eqnarray}
\theta_k \leq \frac{(B_{k}^{-1}\bv, \bv)}{(A_{k}^{\dagger}\bv, \bv)}
=
\frac{\big( \widehat B_{k}^{-1}(Y, P)\bv, (Y, P)\bv \big)}
{\big( \widehat A_{k}^{\dagger}(Y,P)\bv, (Y, P)\bv \big)}\leq 1,\quad(\bv,\one)=0.
\end{eqnarray}
The proof of~\eqref{m-1/a-20} follows from item (ii) in Proposition~\ref{prop:sequence}.
\end{proof}

The spectrum estimate \eqref{m-1/a-1} suggests that, 
$B_{J}^{-1}$ can be used as a preconditioner of 
a Conjugate Gradient method solving a linear system whose coefficient matrix is $A_{J}$.  
It also leads to the following convergence estimate of a power method. 
\begin{theorem}\label{thm_amli_w}
Assume that there is a constant $c_g$ such that $1\leq c_{g} \leq 4$ and 
$
(\widehat A_{k}\bv, \bv) 
\leq (\widehat G_{k}\bv, \bv) 
\leq c_{g} (\widehat A_{k}\bv, \bv)$
for all $\bv$ and $k=2,\ldots,J$. 
Then 
\begin{eqnarray*}
\rho\big( (I-\Pi_{\one})(I-B_{J}^{-1}A) \big) \leq 
\min \left\{\frac{2\sqrt{c}-2}{\sqrt{c}},\frac{2k+\ln k}{2k+\ln
    k+1}\right\} < 1, 
\end{eqnarray*}
where $\Pi_{\one}$ is the $\ell_{2}$ projection to the space of constant vectors. 
\end{theorem}
\begin{proof}
The proof is a directly application of the results in Lemma~\eqref{lem_BkAk}.
\end{proof}

A generalization of this estimate is given by assuming that $c_{g} < m^{2}$ for 
an integer $m$, 
in which case there exists an polynomial $q(t)$ of order $m-1$ such that 
a spectrally equivalent relation can be shown as
\begin{eqnarray*}
\frac{m^{2}-c_{g}}{(m^{2}-1)c_{g}}
\leq 
\frac{(B_{k}^{-1}\bv, \bv)}{(A_{k}^{\dagger}\bv, \bv)}
\leq 
1
, \quad 
\forall \bv:(\bv,\one)=0 \text{ and } k=1,\dots,J , 
\end{eqnarray*}
which then implies that the power method preconditioned by the AMLI method using 
polynomial $q(t)$ on all levels 
has a bounded convergence rate, as 
\begin{eqnarray*}
\rho\big( (I-\Pi_{\one})(I-B_{J}^{-1}A) \big)
\leq 
\frac{m^{2}(c_{g}-1)}{c_{g}(m^{2}-1)}
.
\end{eqnarray*}
For a matching on a hypercubic grid, as discussed above,  
the constant $c_{g}$ 
approaches $4$ asymptotically. 
Assume that the bound is given by $c_{g}=4$, 
then a uniform convergence rate can not be proved by Theorem \ref{thm_amli_w} 
since it requires that the two level spectrally equivalent constants on all levels 
must be 
less or equal to a common bound $c_{g}$ which is strictly less than $4$. 
This suggests us to find the best possible AMLI polynomials 
for the condition $c_{g}=4$, 
and analyze how the AMLI convergence rate relates to the number of levels.

\begin{remark}
An $1-1/\log n$ type convergence rate can also be proven for 
the AMLI methods where the coarse partitioning consists of paths of $m$ vertices 
where $m>2$. 
\end{remark}

\section{Numerical results}\label{numerics}

In the previous section, the 
convergence rate of two-level matching method was used to 
establish the convergence of the matching-based AMLI method. 
Here, a numerical implementation 
that is strictly a translation of this theoretical analysis is considered. 
Then, a simplified and more efficient variant of the method  
is developed and tested.  

To study the effectiveness of the algorithm and the sharpness of the theoretical estimates of its performance derived in the previous section, the method
is applied as a preconditioner to the Conjugate Gradient iteration.
In all tests, the stopping criteria for the PCG solver is set as a $10^{-10}$ reduction in 
the relative $A$ norm of the error.  The average convergence rate, $r_{a}$, and the convergence rates computed by the 
condition number estimates obtained from the Lanczos algorithm and the AMLI polynomial, denoted by $r_{e}$ and $r_{k}$, respectively, are reported.
To reduce the effects of randomness in the numerical results, 
for each combination of testing parameters, 
the PCG method is run for five 
right hand sides computed by random left hand sides, 
and the convergence estimate that represents the worst case is reported.

\subsection{An exact implementation of the AMLI method}

As a first test of the matching AMLI solver, 
it is applied to the graph Laplacian corresponding to 2- and 3-dimensional structured grids
on convex and non-convex domains.  
The coarsening is obtained by applying matching only in a single direction on each level
until the coarsest level is 1-dimensional, 
which is then solved using an LU factorization. 
The AMLI polynomial $q_{k}(t)$ on the $k$-th level is determined by 
the theoretically estimated condition number, 
given by the recursive formula \eqref{recursive-formula-zeta-k}. 
The system $Y_{k}^{T}A_{k}Y_{k}$ is solved exactly by an LU factorization 
on smaller grids or CG iteration down to $10^{-6 }$ relative residual 
on larger grids of the hierarchy. 

Such AMLI method, which is designed to 
have all assumptions in Theorem \ref{thm_amli_w} satisfied, 
is named ``ordinary AMLI method.''
The results are reported in Table \ref{table11} and \ref{table12}
and confirm that the actual convergence rate of the method, $r_{a}$, 
and the condition number estimate, $r_{e}$,
match the theoretical estimate, that is, they both grow in accordance with the estimate 
$r_{k}=(\sqrt{k}-1)/(\sqrt{k}+1)$, where 
$k$ grows logarithmically with respect to the grid size. 

\begin{table}
\captionsetup{width=0.23\textwidth}
\subfloat[Square domain with $n^{2}$ unknowns]
{
\begin{tabular}{|r|cccc|}
\hline
$n$ & $k$ & $r_{k}$ & $r_{e}$ & $r_{a}$ \\
\hline
128 & 13.9 & 0.58 & 0.56 & 0.54 \\
256 & 16.0 & 0.60 & 0.59 & 0.55 \\
512 & 18.0 & 0.62 & 0.58 & 0.57 \\
1024 & 20.1 & 0.64 & 0.60 & 0.60 \\
2048 & 22.1 & 0.65 & 0.61 & 0.61 \\
\hline
\end{tabular}
}
\qquad
\subfloat[L-shaped domain with $(3/4)n^{2}$ unknowns]
{
\begin{tabular}{|r|cccc|}
\hline
$n$ & $k$ & $r_{k}$ & $r_{e}$ & $r_{a}$ \\
\hline
128 & 13.9 & 0.58 & 0.56 & 0.56 \\
256 & 16.0 & 0.60 & 0.57 & 0.59 \\
512 & 18.0 & 0.62 & 0.57 & 0.58 \\
1024 & 20.1 & 0.64 & 0.59 & 0.59 \\
2048 & 22.1 & 0.65 & 0.60 & 0.61 \\
\hline
\end{tabular}
}
\captionsetup{width=0.8\textwidth}
\caption{Results of the AMLI preconditioned CG method 
applied to the graph Laplacians defined on 2D grids.}
\label{table11}
\end{table}
\begin{table}
\captionsetup{width=0.23\textwidth}
\subfloat[Cubic domain with $n^{3}$ unknowns]
{
\begin{tabular}{|r|cccc|}
\hline
$n$ & $k$ & $r_{k}$ & $r_{e}$ & $r_{a}$ \\
\hline
16 & 16.0 & 0.60 & 0.55 & 0.55 \\
32 & 20.1 & 0.64 & 0.59 & 0.59 \\
64 & 24.2 & 0.66 & 0.62 & 0.62 \\
128 & 28.2 & 0.68 & 0.64 & 0.64 \\
\hline
\end{tabular}
}
\qquad 
\subfloat[Fichera domain with $(7/8)n^{3}$ unknowns]
{
\begin{tabular}{|r|cccc|}
\hline
$n$ & $k$ & $r_{k}$ & $r_{e}$ & $r_{a}$ \\
\hline
16 & 16.0 & 0.60 & 0.55 & 0.54 \\
32 & 20.1 & 0.64 & 0.59 & 0.59 \\
64 & 24.2 & 0.66 & 0.62 & 0.62 \\
128 & 28.2 & 0.68 & 0.64 & 0.64 \\
\hline
\end{tabular}
}
\captionsetup{width=0.8\textwidth}
\caption{Results of the ordinary AMLI preconditioned CG method 
applied to the graph Laplacians defined on 3D grids.}
\label{table12}
\end{table}

\subsection{Modified AMLI solver for matching}

Next, a more practical variant of the matching AMLI preconditioner is developed.  
First, the exact $Y_{k}^{T}A_{k}Y_{k}$ solvers are replaced by 
Richardson iterations with weights computed using the
$\ell_{1}$ induced norm of these matrices, 
instead of the common choice of their largest eigenvalues. 

The lower order term $\ln J=\ln \log_{2} n$ in \eqref{recursive-formula-zeta-k} 
is also dropped,
since it is smaller than the term $2J$ in \eqref{recursive-formula-zeta-k} and 
is bounded by $4$ for $n=2^{50}$. 
Another modification to the scheme is the choice of the scaling $\sigma$ 
in Lemma \ref{lem_Ghat} away from $2$. 
Numerical results suggest that $\sigma=2-1/(2 \log_{2} N)$, 
where $N$ is the number of vertices on the graph, 
is usually a better scaling than the estimated bound $\sigma=2$ used in the analysis. 
We use this choice for the 
structured mesh problems and for the unstructured problems 
the scaling is computed through a numerical method.  

In table \ref{table21} and \ref{table22}, 
the convergence rate estimates of this approach applied 
to the same structured problems are reported. 
Although some of the assumptions of the theory are violated by
the method, 
its performance is similar to that of the approach considered in the previous tests.

\begin{remark}
A more practical strategy is to use a numerical method, 
e.g., 
a Lanczos algorithm with an AMLI preconditioner on the $k$-th level, 
to estimate the smallest eigenvalue of $B_{k}^{-1}A_{k}$, 
which is then used to determine the AMLI polynomial on the $k+1$-th level.  
Numerical tests show that 
such strategy results faster convergent AMLI methods than that defined through 
recursive formula \eqref{recursive-formula-zeta-k}, 
at a cost of more complicated setup phase. 
This strategy usually provide a significant speed up 
for 3- or higher dimensional structured problems. 
\end{remark}

\begin{table}
\captionsetup{width=0.23\textwidth}
\subfloat[Square domain with $n^{2}$ unknowns]
{
\begin{tabular}{|r|cccc|}
\hline
$n$ & $k$ & $r_{k}$ & $r_{e}$ & $r_{a}$ \\
\hline
128 & 13.0 & 0.57 & 0.59 & 0.54 \\
256 & 15.0 & 0.59 & 0.62 & 0.58 \\
512 & 17.0 & 0.61 & 0.64 & 0.59 \\
1024 & 19.0 & 0.63 & 0.65 & 0.63 \\
2048 & 21.0 & 0.64 & 0.65 & 0.65 \\
\hline
\end{tabular}
}
\qquad
\subfloat[L-shaped domain with $(3/4)n^{2}$ unknowns]
{
\begin{tabular}{|r|cccc|}
\hline
$n$ & $k$ & $r_{k}$ & $r_{e}$ & $r_{a}$ \\
\hline
128 & 13.0 & 0.57 & 0.58 & 0.56 \\
256 & 15.0 & 0.59 & 0.60 & 0.56 \\
512 & 17.0 & 0.61 & 0.62 & 0.57 \\
1024 & 19.0 & 0.63 & 0.64 & 0.62 \\
2048 & 21.0 & 0.64 & 0.69 & 0.67 \\
\hline 
\end{tabular}
}
\captionsetup{width=0.8\textwidth}
\caption{Results of the modified AMLI preconditioned CG method 
applied to the graph Laplacians defined on 2D grids.}
\label{table21}
\end{table}

\begin{table}
\captionsetup{width=0.23\textwidth}
\subfloat[Cubic domain with $n^{3}$ unknowns]
{
\begin{tabular}{|r|cccc|}
\hline
$n$ & $k$ & $r_{k}$ & $r_{e}$ & $r_{a}$ \\
\hline
16 & 15.0 & 0.59 & 0.50 & 0.42 \\
32 & 19.0 & 0.63 & 0.54 & 0.49 \\
64 & 23.0 & 0.65 & 0.57 & 0.52 \\
128 & 27.0 & 0.68 & 0.59 & 0.56 \\
\hline
\end{tabular}
}
\qquad
\subfloat[Fichera domain with $(7/8)n^{3}$ unknowns]
{
\begin{tabular}{|r|cccc|}
\hline
$n$ & $k$ & $r_{k}$ & $r_{e}$ & $r_{a}$ \\
\hline
16 & 15.0 & 0.59 & 0.49 & 0.49 \\
32 & 19.0 & 0.63 & 0.54 & 0.50 \\
64 & 23.0 & 0.65 & 0.57 & 0.56 \\
128 & 27.0 & 0.68 & 0.57 & 0.60 \\
\hline
\end{tabular}
}
\captionsetup{width=0.8\textwidth}
\caption{Results of the modified AMLI preconditioned CG method 
applied to the graph Laplacians defined on 3D grids.}
\label{table22}
\end{table}

\subsection{On unstructured grids}

Finally, tests of this AMLI preconditioned Conjugate gradient method 
applied to the graph Laplacian defined on
more general graphs, coming from unstructured meshes resulting from triangulations of 
a 2-dimensional grid on a square domain, 
or a 3-dimensional grid on a cubic domain, 
are considered. 
The unstructured grid is generated by 
perturbing grid points of a structured grid 
by a random vector of length $h/2$, 
where $h$ is the mesh size of the original structured grid, 
followed by a Delaunay triangulation. 
Then, a random matching is applied recursively to generate a multilevel hierarchy with 
$(\log_{2}N)/2$ levels. 
The 3-dimensional unstructured grids are generated in a similar way and 
the multilevel hierarchy is constructed accordingly 
by the random matching algorithm. 

The results of these tests are reported in Table \ref{table31} and \ref{table32}.
For the results on the left of these tables, 
the $Y_{k}^{T}A_{k}Y_{k}$ block of the two-level preconditioner is solved 
to high accuracy, 
which is practical 
since this operator is proven well conditioned even for unstructured grids.  
The recursive formula \eqref{tq(t)}
is used to derive the polynomials used in the AMLI cycles, 
and the scaling constants 
are computed using
\begin{eqnarray*}
\sigma_{k}=\max_{i\neq j}\frac{(P_{k}^{T}A_{k}P_{k})_{ij}}{(A_{k+1})_{ij}} , 
\end{eqnarray*} 
which ensures that the upper bound in \eqref{m-1/a-1} is always 1, 
which in turn guarantees that the AMLI method, 
as a preconditioner for the CG method, is always positive semi-definite. 
Because that the AMLI polynomials, constructed according to \eqref{tq(t)}, 
is negative when $t>1$. 
Assume that the scaling constant $\sigma_{k}$ is smaller than the value suggested above, 
then there exists a $\bv$ such that 
$$
(G_{k}^{\dagger}\bv, \bv) > (A_{k}^{\dagger}\bv, \bv) ,
$$
which makes it possible that 
$(B_{k}^{-1}\bv, \bv) > (A_{k}^{\dagger}\bv, \bv)$. 
Assume that happens, 
the matrix $B_{k}^{-1}q(A_{k}B_{k}^{-1})$ becomes indefinite 
which in turn makes $B_{k+1}$ indefinite. 

For the results on the right of Table \ref{table31} and \ref{table32},
the solve of the $Y_{k}^{T}A_{k}Y_{k}$ block is replaced by 
one Richardson iteration, 
and the AMLI polynomials are constructed based on \eqref{tq(t)} 
without the lower order term $\ln k$. 
The asymptotic convergence rates are again
close to the expected convergence rates obtained from the AMLI polynomials
Further, the actual convergence rates are usually better, 
especially for the method that uses more accurate solves for 
the $Y_{k}^{T}A_{k}Y_{k}$ blocks, as opposed to the one
that uses a single Richardson iteration.

\begin{table}
\subfloat[Ordinary AMLI]
{
\begin{tabular}{|r|cccc|}
\hline
$n$ & $k$ & $r_{k}$ & $r_{e}$ & $r_{a}$ \\
\hline
128 & 16.0 & 0.60 & 0.70 & 0.58 \\
256 & 18.0 & 0.62 & 0.72 & 0.54 \\
512 & 20.1 & 0.64 & 0.74 & 0.63 \\
1024 & 22.1 & 0.65 & 0.75 & 0.65 \\
2048 & 24.2 & 0.66 & 0.76 & 0.67 \\
\hline
\end{tabular}
}
\qquad
\subfloat[Modified AMLI]
{
\begin{tabular}{|r|cccc|}
\hline
$n$ & $k$ & $r_{k}$ & $r_{e}$ & $r_{a}$ \\
\hline
128 & 15.0 & 0.59 & 0.70 & 0.70 \\
256 & 17.0 & 0.61 & 0.71 & 0.70 \\
512 & 19.0 & 0.63 & 0.72 & 0.72 \\
1024 & 21.0 & 0.64 & 0.73 & 0.73 \\
2048 & 23.0 & 0.65 & 0.75 & 0.75 \\
\hline
\end{tabular}
}
\caption{Results of the CG method preconditioned by variants of the matching AMLI methods applied to the graph Laplacian defined on 2D unstructured grids of size $n^{2}$.}
\label{table31}
\end{table}

\begin{table}
\subfloat[Ordinary AMLI]
{
\begin{tabular}{|r|cccc|}
\hline
$n$ & $k$ & $r_{k}$ & $r_{e}$ & $r_{a}$ \\
\hline
16 & 18.0 & 0.62 & 0.65 & 0.48 \\
32 & 22.1 & 0.65 & 0.67 & 0.55 \\
64 & 26.2 & 0.67 & 0.70 & 0.62 \\
128 & 30.3 & 0.69 & 0.74 & 0.60 \\
\hline
\end{tabular}
}
\qquad
\subfloat[Modified AMLI]
{
\begin{tabular}{|r|cccc|}
\hline
$n$ & $k$ & $r_{k}$ & $r_{e}$ & $r_{a}$ \\
\hline
16 & 17.0 & 0.61 & 0.59 & 0.55 \\
32 & 21.0 & 0.64 & 0.63 & 0.58 \\
64 & 25.0 & 0.67 & 0.65 & 0.62 \\
128 & 29.0 & 0.69 & 0.67 & 0.65 \\
\hline
\end{tabular}
}
\caption{Results of the CG method preconditioned by 
variants of the matching AMLI methods applied to 
the graph Laplacian defined on 3D unstructured grids of size $n^{3}$.}
\label{table32}
\end{table}

\section{Conclusions}\label{conclusion}

An algebraic formula for estimating the convergence rate of an aggregation-based 
two level method is derived, 
and it is shown that the formula can be used to obtain sharp estimates 
of the convergence rates 
in the special case where matching is used. 
With the use of geometric information, a sharp bound
of the two-level method is derived. 
The nearly optimal convergence and complexity of 
the multilevel method that uses AMLI cycles is also established. 
The reported numerical  tests illustrate the sharpness of the theoretical estimates. 
Moreover all the theoretical results can be generalized to aggregates of general size and, hence, 
can be used to study an approach which combines aggressive aggregation with AMLI cycles, which 
should result in a fast and memory efficient solver for graph Laplacians.  Development and analysis of 
such a scheme and one that uses more general smoothers are subject of on-going research.


\providecommand{\noopsort}[1]{}

\end{document}